\documentclass[a4paper,12pt, oneside]{article}

\usepackage[english]{babel}     
\usepackage[ansinew]{inputenc}   
\usepackage[T1]{fontenc}        
\usepackage{fullpage}
\usepackage{amssymb,amsmath,amsthm,amsfonts,amscd}
\usepackage[dvips]{color}

\usepackage{hyperref}

\newcommand{\Z}{\mathbb{Z}}       
\newcommand{\Q}{\mathbb{Q}}       
\newcommand{\R}{\mathbb{R}}
\newcommand{\C}{\mathbb{C}}       

\newcommand{\ord}{\mathrm{ord}\,}

\theoremstyle{plain}
\newtheorem{lause}{Theorem}         
\newtheorem{lemma}{Lemma}[section]

\theoremstyle{definition}              

\theoremstyle{remark}

\title{Baker-Type Estimates for Linear Forms in the Values of $q$-series}
\author{Leena Leinonen\footnote{The author was supported by Finnish Academy of Science and Letters, V\"ais\"al\"a Foundation, 
and by the Academy of Finland, grant 138522.}}

\begin{document}

\maketitle

\begin{abstract}
A Baker-type linear independence measure is obtained for the values of generalized Heine series at non-zero points of an imaginary quadratic number field. This kind of estimate depends on the individual coefficients of the linear form, not only on the maximum of their absolute values. 
\end{abstract}

\maketitle

In the present paper we are interested in a certain type linear independence measures for the values of generalized Heine series. For this purpose we define these functions precisely first.
Let $K$ denote the field of the rational numbers $\Q$ or  an imaginary quadratic number field $\mathbb{I}$,  and  $\Z_{K}$ its ring of integers. 
Let $q=\frac{a}{b}\in K$ ($a,b\in \Z_{K}$) be such that $|q| > 1$. The generalized Heine series are defined by the equation
\begin{equation}\label{heine}
\varphi(z)= 1 + \sum_{n=1}^\infty \frac{q^{-sn(n-1)/2}}{P(1)P(q^{-1})\cdots P(q^{-(n-1)})}z^n,
\end{equation}
where $s\in \Z_+$,\, 
$P(z)$ is a polynomial in $K[z]$, $\deg P(z)\leq s$, $P(0)\ne 0$ and $P(q^{-k})\ne 0$ for all $k=0,1,2,\ldots$.
There are two interesting special cases of \eqref{heine},
the Tschakaloff function ($s=1$, $P(z) \equiv q$)
\begin{equation}\label{Tsch}
T_q(z)= \sum_{n=0}^\infty q^{-n(n+1)/2}z^n
\end{equation}
and the $q$-exponential function ($s=1$, $P(z)=q-z$)
\begin{equation}\label{exp_qq}
E_q(z)=\sum_{n=0}^\infty \frac{z^n}{(q-1)\cdots(q^n-1)}=\prod_{n=1}^\infty(1+
\frac{z^n}{q^n}).
\end{equation}
The arithmetical properties of \eqref{Tsch} and \eqref{exp_qq} have been studied in numerous works. There are two excellent surveys, \cite{Bu} and \cite{Va}, concerning these results. Furthermore, the arithmetical properties of the generalized Heine series \eqref{heine} have been studied in several papers. 
For example, Stihl \cite{ST} considered the linear independence of the values of \eqref{heine} in the case
\[P(z)= (1-q^{-\beta_1}z)(1-q^{-\beta_2}z)\cdots(1-q^{-\beta_N}z)(1-a_{N+1}z)(1-a_{N+2}z)\cdots (1-a_tz)\]
with $\beta_i\in \Z_+$, $a_i\in \Q$ and $0< N < t <s$. Stihl's  result have been generalized in \cite{Ka} and \cite{SV}, but in all these works it has been assumed that $t<s$. Bézivin \cite{Be} considered also the case $t=s$, but his linear independence result was not quantitative. The first linear independence measure for the values of generalized Heine series \eqref{heine} was presented in \cite{Va2} by Väänänen.

Typically linear independence measures are given in the terms of maximum of the coefficients of the linear form. For example, the estimates in \cite{Ka}, \cite{SV}, \cite{ST} and \cite{Va2} are in this form. Baker, in 1956, introduced in \cite{Ba} a different type linear independence measure, where the measure depends on individual coefficients. He proved this kind of measure for the values of exponential function. 

There are several later Baker-type linear independence results, see eg.\cite{Fe1}, \cite{Fe2}, \cite{Mah}, \cite{So}, \cite{Va1} and \cite{Zu}. Matala-aho \cite{Mat} has very recently made an axiomatic approach to Baker-type estimates.

The first Baker-type measure for $q$-series was obtained by Väänänen and Zudilin \cite{VZ}. They studied the values of 
\eqref{heine} in the case where $q\in \Z_{K}$.  We are going to generalize this result for all $q\in K$ which are nearly integers (see condition \eqref{upp_gamma}).
We shall follow the construction idea of \cite{VZ} with some refinements to our case. 
More precisely, we prove the following result.
\begin{lause}\label{thm1}
Suppose that $\alpha_1, \ldots, \alpha_m$ are non-zero elements of $K$ satisfying conditions \,$\alpha_i \ne \alpha_jq^{\,l}$, $l\in \Z$,\, for all $i \ne j$\, and either $\deg P(z) < s$ or $\deg P(z)=s$ and $\alpha_i \neq P_sq^n$ ($i=1,\ldots, m$, $n\in\Z_+$), where $P_s$ is the leading coefficient of $P(z)$. Suppose
that
\begin{equation}\label{upp_gamma}
\gamma := \frac{\log|b|}{\log|a|} < \Gamma(m,s),
\end{equation}
where $\Gamma(m,s)$ is defined below in \eqref{upp_Gamma}.  
Then for any given $\varepsilon >0$, there exists a positive constant $C=C(\varepsilon)$ such that for all \, $(l_0,l_1,\ldots, l_m)\in \Z_K^{m+1} \setminus\{\bar{0}\}$ 
we have 
\begin{equation}
|l_0 + l_1\varphi(\alpha_1)+ \cdots + l_m\varphi(\alpha_m)| > CH^{-\mu-\varepsilon},
\end{equation}
where
\[ H = \prod_{i=1}^m H_i, \quad H_i = \max\{1, |l_i|\}\] 
and
\[\mu=\mu(m,s, \gamma)= \frac{4s\rho_0^2 + 4(s+2)\rho_0 + (s+17)}{4\rho_0 - 13m-\gamma(4ms\rho_0^2+4(ms+2m+1)\rho_0 +ms+ 4m + 2)}\]
with
\[\begin{split}
\rho_0 = &\rho_0(m,s, \gamma) =\frac{13m}{4} +\frac{\gamma}{2(1-\gamma)}\\
&+ \sqrt{\left( \frac{13m}{4}+ \frac{\gamma}{2(1-\gamma)}\right)^2 + \frac{13m(s+2) +s +17}{4s}+\frac{(s+2)\gamma}{2s(1-\gamma)}}.
\end{split}\]

Define that 
\[f(\gamma)= \frac{4\rho_0 - 13m}{4sm\rho_0^2 + 4(2m+ms+1)\rho_0 + 4m+ms+2}-\gamma.\] The upper bound in the condition \eqref{upp_gamma} is 
\begin{equation}\label{upp_Gamma}
\Gamma(m,s)= \min \{\tau\,| \, f(\tau)=0,\, 0<\tau< 1\}. \end{equation}
\end{lause}

When $\gamma=0$ (or $q\in \Z_K$), our result is exactly the same as in \cite{VZ}. Our general result can be also obtained  by using Matala-aho's axiomatic approach \cite{Mat}, which gives us the error term $\varepsilon$ more accurately:
\[\varepsilon=\frac{A}{\sqrt{\log H}},\]
where the constant $A$ can be computed explicitly. We will also give an alternative proof to our result, where we apply Matala-aho's method.


\section{A Difference Equation} 
\label{difference}
We shall consider analytic solutions of the Poincaré-type $q$-difference equation
\begin{equation}\label{differenssi}
\alpha z^s f(z)= P(z)f(qz) + Q(z),
\end{equation}
where $\alpha\in K$ is non-zero, $s\in \Z_+$,  $P(z)$ and $Q(z)$ are polynomials in $K[z]$,  $t=\deg{P(z)} \leq s$, $P(0)\neq 0$ and $Q(z)\not\equiv 0$. Amou, Katsurada and Väänänen 
introduced the connection between the equation \eqref{differenssi} and the function $\varphi(z)$ 
already in  \cite{AKV}. We show similarly as they did in \cite{AKV} that \eqref{differenssi} 
has a unique solution in the set of formal power series $\C[[z]]$, which converges in a 
neighbourhood of the origin. After that we will reveal the connection between the solution and $\varphi(z)$. 

Let 
\[f(z) = \sum_{\nu=0}^\infty f_\nu z^\nu \in \C[[z]]\] 
be a solution of \eqref{differenssi}. Let us denote
$P(z) = \sum_{i=0}^t P_i z^i$, \, $Q(z)= \sum_{i=0}^u Q_iz^i$.
Then  using \eqref{differenssi} we obtain
\[\alpha f_{\nu-s} = \sum_{i=0}^t P_iq^{\nu-i}f_{\nu-i} + Q_\nu,\]
where $f_\nu = 0$ for all $\nu < 0$ and $Q_\nu =0$ for all $\nu > u$. 
Hence we get a recursion formula 
\begin{equation}\label{coeff_recur}
P_0q^\nu f_\nu = \alpha f_{\nu-s} - \sum_{i=1}^t P_iq^{\nu-i}f_{\nu-i} - Q_\nu,
\end{equation} 
which defines the coefficients $f_\nu$ uniquely.
From \eqref{coeff_recur} with $|q|> 1$ it follows that
\[|f_\nu|  \leq C_1 \max\{1, |f_0|, \ldots, |f_{\nu-1}|\},\]
where $C_1$ (as $C_2$, $C_3$, \ldots later) is a positive constant depending only on $s, q, \alpha, P(z)$ and $Q(z)$.
This implies that 
\begin{equation} \label{upper_f}
|f_\nu| \leq C_1^{\nu + 1}
\end{equation}
and consequently  $f(z)$ converges in  a neighbourhood of the origin. 
We denote by $D$ the disk in $\C$, where $f(z)$
converges. Next we show that by using \eqref{differenssi} repeatedly $f(z)$ can be continued meromorphically beyond $D$, to whole $\C$.
Let $\mathcal{P}$ be the set defined by 
\[\mathcal{P} = \{q^k\beta | P(\beta) =0, k\geq 1\}.\]
Using \eqref{differenssi} we get
\[
f(z)
    = \left(\prod_{i=1}^k \frac{\alpha(q^{-i}z)^s}{P(q^{-i}z)}\right) f(q^{-k}z) - \sum_{i=1}^k \left(\prod_{j=1}^{i-1}\frac{\alpha(q^{-j}z)^s}{P(q^{-j}z)}\right) \frac{Q(q^{-i}z)}{P(q^{-i}z)}
\]
for any $z\in (\C\setminus \mathcal{P}) \cap D$. Since the first product tends to zero as $k \to \infty$, we get
\begin{equation}\label{diffsol}
f(z)=-\sum_{i=1}^\infty \frac{q^{-si(i-1)/2}Q(zq^{-i})}{P(zq^{-1})\cdots P(zq^{-i})}(\alpha z^s)^{i-1}.
\end{equation}
If we choose $Q(z) = -P(z)$, we get 
$f(q)=\varphi(\alpha)$. Thus we can consider the linear independence of $\varphi(\alpha_1),\ldots, \varphi(\alpha_m)$ by considering a system of difference equations of type \eqref{differenssi}.

In the next section we use Thue-Siegel lemma to get Padé-type approximations of the second kind for the functions $f_i(z)$. We need equations with integer coefficients and therefore we define 
\begin{equation}\label{iso-F}
F_{\nu} = P_0^{\nu+1}q^{\nu(\nu+1)/2}f_{\nu}.
\end{equation}
From  the recurrence formula \eqref{coeff_recur}  and the definition \eqref{iso-F} we see that
\begin{equation}\label{deg_F}
\begin{split}
 &F_\nu \in \Z[\alpha, P_0, \ldots, P_t, Q_0,\ldots, Q_u, q], \\
\deg_{\alpha}F_{\nu} \leq \nu, \quad 
&\deg_{P_0,\ldots,P_t, Q_0,\ldots, Q_u} F_\nu \leq \nu +1, \quad 
\deg_q F_\nu \leq \nu(\nu+ 1)/2.
\end{split}
\end{equation}

In the third section we construct more Padé-type approximations by an iteration process. In that process we need the
following lemma, which was already presented in  \cite[(11)]{VZ}. 
\begin{lemma} \label{iteration}
The iteration equation

\[(\alpha z^s)^kq^{uk} f(zq^{-k}) = X_k(z,q)f(z) + Y_k(z,q),
\]
where 
\[ X_k(z,q)= q^{sk(k+1)/2 +uk}\prod_{j=1}^k P(zq^{-j})\]
(is independent of $\alpha$ and $Q(z)$) and
\[ Y_k(z,q)= \sum_{j=1}^k (\alpha z^s)^{j-1}q^{s(k(k+1)/2-j(j-1)/2)+uk}Q(zq^{-j})\prod_{l=j+1}^k P(zq^{-l})\]
holds.
\end{lemma}
Lemma \ref{iteration} is a direct corollary of the difference equation \eqref{differenssi}. Further it  
gives us the upper bound
\begin{equation}\label{bound_iterx}
|X_k(z,q)|\le C_2^k|q|^{sk(k+1)/2}\max\{1, |z|\}^{C_3 k}.
\end{equation}


\section{Pad\'e-type Approximations}

Let $\alpha_1, \ldots, \alpha_m\in K\setminus\{0\}$ and consider a system of difference equations
\begin{equation}\label{differenssit}
\alpha_iz^sf_i(z)= P(z)f_i(qz) + Q_i(z), \quad i=1,\ldots, m.
\end{equation}
 Let
\[f_i(z)=\sum_{\nu=0}^\infty f_{i\nu}z^\nu, \quad i=1,\ldots, m,\]
be the analytic solution of \eqref{differenssit}. We shall construct Padé-type approximations of the second kind for these functions similarly as in \cite{VZ}.

Let $n_1,\ldots, n_m$ be positive integers and $N=n_1+\ldots + n_m$. Let us choose such $\delta$, which satisfies \, $0< \delta < 1/m$ \, and
\begin{equation}
n_i \geq \delta N, \quad i=1,\ldots, m.
\end{equation}
We are looking for a polynomial
\begin{equation} \label{apppol}
A(z) = \sum_{\mu=0}^N \frac{a_\mu}{q^{\mu(\mu-1)/2}}z^\mu\not\equiv 0
\end{equation}
with integer coefficients $a_\mu\in\Z_K$ such that for all $i=1,\ldots, m$ the expansion 
\[A(z)f_i(z) = \sum_{k=0}^\infty b_{ik}z^k\]
satisfies the conditions $b_{ik}=0$ for $k=N+1, N+2,\ldots, N+n_i-[\delta N]-1.$ 
We have 
\[\begin{split}
A(z)f_i(z)  
			&= \sum_{k=0}^\infty \sum_{\substack{\nu=0\\ \nu \geq k-N}}^k \frac{f_{i\nu}a_{k-\nu}}{q^{(k-\nu)(k-\nu-1)/2}}z^k\\
           &= \sum_{k=0}^\infty \sum_{\substack{\nu=0 \\ \nu \geq k-N}}^k \frac{F_{i\nu}a_{k-\nu}}{P_0^{\nu+1}q^{k(k-1)/2 + \nu(\nu-k+1)}}
z^k,\end{split}\]
where analogously to \eqref{iso-F} $F_{i\nu}= P_0^{\nu+1}q^{\nu(\nu+1)/2} f_{i\nu}$. Thus the condition $b_{ik}=0$ for $k>N$ is equivalent to
\begin{equation}\label{toPade}
\sum_{\nu=k-N}^k P_0^{k-\nu}q^{(\nu+1)(k-\nu)}F_{i\nu}a_{k-\nu}=0.
\end{equation}
We choose now natural numbers $A$ and $B$ to be such that \,$A\alpha_i\in \Z_{K}$ \, and \, $BP(z)$, $BQ_i(z) \in \Z_{K}[z]$.
Due to  \eqref{deg_F} we get a linear equation in $a_\mu$ which has integer coefficients from $K$, if we multiply the equation \eqref{toPade} by $(AB^2)^kb^{k(k+1)/2}$. Using \eqref{upper_f} and \eqref{iso-F} we see that the integer coefficients of this equation  satisfy the condition
\[\begin{split}
|coeff| &\leq C_4^k |b|^{k(k+1)/2} \max_{k-N\leq \nu \leq k}\{|q|^{(\nu+1)(k-\nu)+\nu(\nu+1)/2}\}\leq C_4^k |bq|^{k(k+1)/2},
\end{split}\]
which can be written in the form
 \[ |coeff| \leq C_5^k |a|^{k^2/2}.\]
We need the condition $b_{ik}= 0$ for $k=N+1, N+2, \ldots, N+ n_i-[\delta N] -1$ and for these $k$ we have
\[\begin{split}
k &\leq N+ n_i -\delta N = N + (N- n_1-\cdots-n_{i-1}-n_{i+1}-\cdots-n_m) -\delta N \\
  &\leq 2N -m\delta N.
\end{split}\]
Thus the absolute values of coefficients are bounded by
\[C_6^N |a|^{(2N- m\delta N)^2/2}.\]

In order to get the Padé-type approximations we need to use Thue-Siegel lemma, which is  presented below and proved in \cite[Chapter 3, Lemma 13]{Sh}.

\begin{lemma}[Thue-Siegel lemma]
Let $B,M\in \Z$ satisfy the condition $B>M>0$. Furthermore, let 
$a_{ij}\in \Z_K$, $1\le i \le M, 1\le j\le B$ satisfy the condition  
$|a_{ij}|\le U$ with $U \in \R_+$. Then there exists a non-trivial $x_1,\ldots, x_B\in \Z_K$ such that 
\[\sum_{j=1}^B a_{ij}x_j=0, \quad i=1,\ldots,M,\]
and
\[|x_j|\le c_K(c_K BU)^{\frac{M}{B-M}}, \quad j=1,\ldots,B,\]
where $c_K$ is a constant depending only on $K$.
\end{lemma}

The number of the linear equations $b_{ik}=0$ is equal to 
\[\sum_{i=1}^m(n_i -[\delta N] -1) = N -m([\delta N] +1)\]
and the number of indeterminates $a_\mu$ is $N+1$. 
Hence Thue-Siegel lemma 
yields the existence of integers $a_\mu \in \Z_K$, not all zero, such that
\begin{equation}\label{coeff_bound}
|a_\mu| 
        \leq C_7^N |a|^{\gamma_1 N^2}, \quad \gamma_1=\frac{(2-m\delta)^2(1-m\delta)}{2m\delta}.
\end{equation}

\vspace{2 em}

Let us define $B_i(z)=\sum_{k=0}^N b_{ik}z^k$ and $R_i(z)= \sum_{k=N+n_i-[\delta N]}^\infty b_{ik}z^k$. 
If $k\leq N$, then
\[b_{ik} = \sum_{\nu=0}^k \frac{F_{i\nu}a_{k-\nu}}{P_0^{\nu+1}q^{k(k-1)/2 + \nu(\nu-k+1)}}=\frac{1}{q^{k(k+1)/2}}\sum_{\nu=0}^k \frac{F_{i\nu}a_{k-\nu}q^{(\nu+1)(k-\nu)}}{P_0^{\nu+1}}.\]       It follows that polynomials 
\[a^{N(N+1)/2}(AB^2P_0)^{N+1}B_i(z)\]
have integer coefficients in  $K$.
%
By \eqref{upper_f} and \eqref{coeff_bound}, for all $k> N$, the following estimates hold
\[\begin{split}
|b_{ik}|&= \left|\sum_{\nu=k-N}^k \frac{f_{i\nu}a_{k-\nu}}{q^{(k-\nu)(k-\nu-1)/2 }}\right|\leq C_1^{k+1}C_7^N |a|^{\gamma_1 N^2}\sum_{\nu=k-N}^k\frac{1}{|q|^{(k-\nu)(k-\nu-1)/2}}\\
        &\leq C_8^k|a|^{\gamma_1 N^2}.
\end{split}\]
We obtain for all $|z|<(2C_8)^{-1} $ that
\begin{equation}\label{reminder}
|R_i(z)| \leq |\sum_{k=N_i}^\infty b_{ik}z^k| \leq |a|^{\gamma_1 N^2} \sum_{k=N_i}^\infty (C_8|z|)^k 
\leq C_{9}^N |a|^{\gamma_1 N^2}|z|^{N_i},
\end{equation}
where $N_i= N+ n_i - [\delta N]$.
We have thus proved the following lemma.
\begin{lemma}\label{pade-poly}
There exists a polynomial 
\[A(z)=\sum_{\mu=0}^N \frac{a_\mu}{q^{\mu(\mu-1)/2}}z^\mu\]
with  $a_\mu\in \Z_{K}$ satisfying 
\[|a_{\mu}|\leq C_7^N|a|^{\gamma_1 N^2}, \quad \gamma_1 =\frac{(2-m\delta)^2(1-m\delta)}{2m\delta}\]
 and  polynomials 
\[B_i(z)=\sum_{k=0}^N b_{ik}z^{k}\]
such that
\[A(z)f_i(z) - B_i(z) = 
R_i(z),\]
where the forms $R_i(z)$ satisfy the conditions $\ord_{z=0}R_i(z)\ge N+n_i-[\delta N]$ and
\[|R_i(z)|\leq C_{9}^N |a|^{\gamma_1 N^2}|z|^{N_i}\]
for all $|z|< (2C_8)^{-1}$. 
Further the polynomials
\[ a^{N(N-1)/2}A(z), \quad  a^{N(N+1)/2}(AB^2P_0)^{N+1}B_i(z)\]
have integer coefficients. 
\end{lemma}


\section{Iteration Process}

Let us denote now
\[A_0(z)= A(z), \quad B_{0i}(z) = B_i(z), \quad R_{0i}(z) = R_i(z).\]
Due to Lemma \ref{pade-poly} we have the equation
\[A_0(z)f_i(z)-B_{0i}(z)= R_{0i}(z).\]
If we operate this equation by the $q$-sift operator $J_z$ (where $J_zf(z)=f(qz)$) and 
after that apply  the $q$-difference equation \eqref{differenssit}, we get the equality
\[A_0(qz)(\alpha_iz^s f_i(z)-Q_i(z)) -P(z)B_{0i}(qz)= P(z)R_{0i}(qz),\]
which implies immediately
\[z^sA_0(qz)f_i(z) - \alpha_i^{-1}(A_0(qz)Q_i(z)+P(z)B_{0i}(qz))=\alpha_i^{-1}P(z)R_{0i}(qz).\]

Starting from $A_0(z)$  and $B_{0i}(z)$ 
we build further Padé-type approximations by iterating the above process (similarly as in \cite{VZ}). We can define
\begin{equation}\label{rec1}
\begin{split}
A_j(z)&= z^s A_{j-1}(qz),\\ 
B_{ji}(z)&= \alpha_i^{-1}(A_{j-1}(qz)Q_i(z)+ P(z)B_{j-1,i}(qz)), 
\end{split}
\end{equation}
where $i=1,\ldots, m$, $j=1,2,\ldots$ . 
Then the equation
\begin{equation}\label{rec2}
A_j(z)f_i(z) - B_{ji}(z)= R_{ji}(z)
\end{equation}
holds for all $i=1,\ldots, m$, $j=0,1,\ldots$,
where
\begin{equation}\label{rec3}
R_{ji}(z) =\alpha_i^{-1}P(z)R_{j-1,i}(qz).
\end{equation}

Next we will consider the determinant
\begin{equation} \label{determ}
\begin{split}
\Delta(z) &= \det\left( \begin{matrix}A_0(z) & B_{01}(z)& \cdots & B_{0m}(z)\\
                                      A_1(z) & B_{11}(z) & \cdots & B_{1m}(z)\\                                      \cdots & \cdots    & \cdots   &  \cdots\\                                      A_m(z) & B_{m1}(z) & \cdots & B_{mm}(z)\end{matrix}\right)
\\ 
&= (-1)^m\det\left(\begin{matrix}A_0(z) & R_{01}(z) & \cdots & R_{0m}(z)\\
                                           A_1(z) & R_{11}(z) & \cdots & R_{1m}(z) \\
                                           \cdots & \cdots & \cdots & \cdots \\
                                           A_m(z) & R_{m1}(z) & \cdots & R_{mm}(z) \end{matrix}\right).
\end{split}\end{equation}

\begin{lemma}
Suppose  that  none of the functions  $f_i(z)$ is a polynomial and $\alpha_i\ne \alpha_jq^l$  for all $i \ne j$ and $l\in \Z$.   Then  $\Delta(z) \not\equiv 0$ (see \cite[Lemma 3]{Va2}).
\end{lemma}
\begin{proof}[Proof]
If $z_0$ is an element at which $f_i(z)$ does not have a limit, then $f_i(z)$ does not have a limit at points $z= q^kz_0$ too, since $f_i(z)$ can be expressed in the form \eqref{diffsol}. This implies that $f_i(z)$ has either none or infinitely many poles.  
Hence if $f_i(z)$ is a rational function it has to be a polynomial.

Because we assumed that $f_i(z)$ is not a polynomial, it is neither a rational function and hence  $R_{0i}(z)\not\equiv 0$ $(i=1,...,m)$.
 Let 
\[az^k, \quad r_iz^{k_i}\quad (ar_1\cdots r_m \neq 0)\] 
denote the lowest degree terms of $A(z)$ and $R_{0i}(z)$. We also know that $P(0)=P_0\neq 0$.
From the expression of $\Delta(z)$ in \eqref{determ} and the recursive formulae \eqref{rec1} 
 and \eqref{rec3} 
 it follows that
\[\ord \Delta(z) \geq k+ k_1 +\ldots + k_m.\]
By using the determinant of the coefficients of these lowest degree terms
\[\begin{split}
&\left|\begin{matrix} a & r_1 & \cdots & r_m\\
                     0 & {\alpha_1}^{-1}P_0r_1q^{k_1} & \cdots & {\alpha_m}^{-1}P_0r_mq^{k_m} \\
                      \cdots                           &         &                             \\
                      0 & {\alpha_1}^{-m}P_0^m r_1 q^{mk_1} & \cdots & {\alpha_m}^{-m}P_0^m r_m q^{mk_m} \end{matrix}\right|\\
                      &= aP_0^{m(m+1)/2}\prod_{i=1}^m(\alpha_i^{-1}r_iq^{k_i})
                      \left|\begin{matrix} 1 &  \cdots & 1 \\
                                           \alpha_1^{-1}q^{k_1} & \cdots & \alpha_m^{-1}q^{k_m} \\
                                           \cdots & & \\
                                           (\alpha_1^{-1}q^{k_1})^{m-1} & \cdots & (\alpha_m^{-1}q^{k_m})^{m-1} \end{matrix}\right|\\
                      &= aP_0^{m(m+1)/2}\prod_{i=1}^m(\alpha_i^{-1}r_iq^{k_i})\prod_{1\leq i <j\leq m}(\alpha_j^{-1}q^{k_j} - \alpha_i^{-1}q^{k_i})
\end{split}\]
 we see that the coefficient of $z^{k+ k_1 + \ldots + k_m}$ in $\Delta(z)$ is non-zero, since $\alpha_i/\alpha_j \neq q^l$, $l\in \Z$.   This proves our lemma.                   
\end{proof}

Since by Lemma \ref{pade-poly}, \eqref{rec1}, \eqref{rec3} and \eqref{determ}  we have
\[\ord \Delta(z) \geq N_1 + \ldots + N_m = \sum_{i=1}^m (N+ n_i-[\delta N]) 
\geq (m+1)N -m\delta N \]
and 
\[
\deg \Delta(z) 
\leq (m+1)N + S\frac{m(m+1)}{2},
\]
where $S=\max\{s, \deg Q_i(z)\}$, we deduce that $\Delta(z)$ has at most
\[\deg \Delta(z) - \ord \Delta(z) \leq m\delta N + S\frac{m(m+1)}{2}\]
non-zero zero-points. If we take any $\alpha \in K\setminus \{0\}$, then
for each $\rho > m\delta$ there exists an integer $k$
satisfying
\begin{equation}\label{k1}
(\rho-m\delta)N - S\frac{m(m+1)}{2} -1 < k \leq \rho N
\end{equation}
 and
\begin{equation}\label{k2}
\Delta(\alpha q^{-k})\ne 0.
\end{equation}

Let $\alpha\in K\setminus \{0\}$ be an element which satisfies the condition $P(\alpha q^{-k}) \neq 0$ for $k=1,2,...$.
Then $f_i(\alpha q^{-k})$ is defined for all $i=1,...,m$ and $k=0,1,\ldots$.
Let us  denote   
$u= \max_{1\le i \le m}\{\deg Q_i(z)\}$ and define
\begin{equation}\label{r}
\hat{r}_{jik} = (\alpha_i\alpha^s)^kq^{uk}R_{ji}(\alpha q^{-k}).
\end{equation}
By using \eqref{rec2} and Lemma \ref{iteration} we obtain
\begin{equation}
\begin{split}
\hat{r}_{jik} 
             &= (\alpha_i\alpha^s)^kq^{uk}\left(A_j(\alpha q^{-k})f_i(\alpha q^{-k}) -B_{ji}(\alpha q^{-k})\right)\\
             &= A_j(\alpha q^{-k})\left( X_k(\alpha,q) f_i(\alpha)+ Y_{ik}(\alpha,q)\right) - (\alpha_i\alpha^s)^kq^{uk}B_{ji}(\alpha q^{-k})\\
             &=\hat{p}_{jk}f_i(\alpha)-\hat{q}_{jik},
\end{split}
\end{equation}
where 
\begin{equation}\label{hat_p}
\hat{p}_{jk}= A_j(\alpha q^{-k})X_k(\alpha,q)
\end{equation}
and
\begin{equation}\label{hat_q}
\hat{q}_{jik}= (\alpha_i\alpha^s)^kq^{uk}B_{ji}(\alpha q^{-k})- A_j(\alpha q^{-k})Y_{ik}(\alpha,q).
\end{equation}
Next we try to find such $D_k\in \Z_K$ that $D_k\hat{p}_{jk}$ and $D_k\hat{q}_{jik}$ are integers in $K$.

From the  recursion formulae \eqref{rec1} we get
\begin{equation}\label{D_k1}
A_j(z) 
= q^{sj(j-1)/2}z^{js}A_0(q^jz)
\end{equation}
and
\[\begin{split}
B_{ji}(z)
         = \left(\prod_{l=0}^{j-1}\alpha_i^{-1}P(q^{l}z)\right)&B_{0i}(q^jz)\\ &+ \sum_{l=1}^j \left(\prod_{n=0}^{l-2}P(q^nz)\right)\alpha_i^{-l}Q_i(q^{l-1}z)A_{j-l}(q^lz).
\end{split}
\]
Due to Lemma \ref{pade-poly}  polynomials \, $a^{N(N-1)/2}A_{0}(z)$ \, and \, $a^{N(N+1)/2}(AB^2P_0)^{N+1}B_{0i}(z)$ \, have integer coefficients and 
$\deg A_0(z), \deg B_{0i}(z) \le N$.
Let $A_1$ 
be a non-zero rational integer 
such that \,$A_1 \alpha_i^{-1}$ 
is an integer 
in $K$, and $S= \max\{s, u\}$. 
Hence, if we choose 
\begin{equation}\label{D}
D= a^{N(N+1)/2}b^{Sm(m-1)/2 + mN}(AB^2P_0)^{N+1}(A_1B)^m,
\end{equation}
then polynomials $DA_j(z)$ and $DB_{ji}(z)$ have integer coefficients in $K$ for all $j=0,\ldots, m$.
Let $A_2$ be a such non-zero rational integer  that $A_2\alpha$ is an integer in $K$.
If we choose now 
\begin{equation}\label{D_k}
\hat{D}_k= D(a^kA_2)^{mS+N},
\end{equation}
then the numbers
\[\hat{D}_kA_j(\alpha q^{-k}), \quad  \hat{D}_k B_{ji}(\alpha q^{-k})\]
  are integers in $K$ for all $j=0,\ldots,m$ and $k=0,1\ldots$.
In addition Lemma \ref{iteration} implies that
\[b^{sk(k+1)/2+uk} (AB A_2^{S})^k\left( X_k(\alpha,q), Y_{ik}(\alpha,q)\right) \in \Z_{K}^2.\]
Hence we obtain the following lemma.
\begin{lemma}\label{last_app}
If we choose
\[D_k = b^{sk(k+1)/2+uk}(ABA_2^S)^{k}\hat{D}_k\]
and define
\[
r_{jik}= D_k \hat{r}_{jik}
\]
then 
\[r_{jik}= p_{jk}f_i(\alpha)-q_{jik},\]
where $p_{jk}$ and $q_{jik}$ are integers in $K$.
\end{lemma}

\vspace{0.5 em}

\begin{lemma} \label{determ_cond}
Assume that $k$ satisfies the conditions \eqref{k1} and \eqref{k2}.
We have 
\begin{equation}\label{determinantcond}
\Delta_k := \det\left(\begin{matrix} p_{0k} & q_{01k} & \cdots & q_{0mk}\\ 
                           p_{1k} & q_{11k} & \cdots & q_{1mk} \\
                           \cdots \\
                           p_{mk} & q_{m1k} & \cdots & q_{mmk}\end{matrix}\right) 
                           \neq 0.
\end{equation}                           
\end{lemma}

\begin{proof}[Proof]
Due to Lemma \ref{last_app} we have  $r_{jik}= p_{jk}f_i(\alpha) - q_{jik}$, where  $r_{jik}= D_k \hat{r}_{jik}$ and $p_{jk}= D_k\hat{p}_{jk}$. Hence we get
\[\begin{split}
\Delta_k &= 
                           (-1)^m D_k^{m+1}
         \det\left(\begin{matrix} \hat{p}_{0k} & \hat{r}_{01k} & \cdots & \hat{r}_{0mk}\\ 
                           \hat{p}_{1k} & \hat{r}_{11k} & \cdots & \hat{r}_{1mk} \\
                           \cdots \\
                           \hat{p}_{mk} & \hat{r}_{m1k} & \cdots & \hat{r}_{mmk} \end{matrix}\right)                         
                           .
\end{split}
\]
Further by \eqref{r} and \eqref{hat_p} we 
obtain that
\[\begin{split}
\Delta_k
= (-1)^m D_k^{m+1} & X_k(\alpha,q)\prod_{i=1}^m(\alpha_i\alpha^sq^u)^k
\\ & \cdot \det\left(\begin{matrix}  A_0(\alpha q^{-k}) & R_{01}(\alpha q^{-k}) & \cdots & R_{0m}(\alpha q^{-k})\\ 
                           A_1(\alpha q^{-k})  & R_{11}(\alpha q^{-k}) & \cdots & R_{1m}(\alpha q^{-k}) \\
                           \cdots \\
                          A_m(\alpha q^{-k})  & R_{m1}(\alpha q^{-k}) & \cdots & R_{mm}(\alpha q^{-k}) \end{matrix}\right).
\end{split}\]
Due to \eqref{determ}                            
\[ \Delta_k 
= D_k^{m+1} X_k(\alpha, q) (q^{u}\alpha^{s})^{mk}\Delta(\alpha q^{-k})\prod_{i=1}^m \alpha_i^k,                                               
\]
which is non-zero.
\end{proof}

\vspace{1 em}

Finally we approximate the values of $|p_{jk}|$ and $|r_{jik}|$.
By the equality \eqref{D_k1} and Lemma \ref{pade-poly}  we get an upper bound
\begin{equation}\label{bound_A}
\begin{split}
|A_j(\alpha q^{-k})|&\leq |q|^{sm(m-1)/2}\max\{1, |\alpha|\}^{mS}|A_0(q^{j-k}\alpha)| 
\\			&\leq |q|^{sm(m-1)/2 + mN}\max\{1, |\alpha|\}^{mS+ N} (N+1) C_7^N |a|^{\gamma_1 N^2}\\
                   &\leq C_{10}^N|a|^{\gamma_1 N^2}
\end{split}
\end{equation}
for all $k=0,1,\ldots$ and $j=0,\ldots,m$.
Furthermore by the equation \eqref{rec3} and Lemma \ref{pade-poly} we get an upper bound for  reminder terms 
\begin{equation}\label{bound_rem}
\begin{split}
|R_{ji}(\alpha q^{-k})| &\leq |\alpha_i^{-j}P(\alpha q^{-k})P(\alpha q^{-k+1})\cdots P(\alpha q^{-k+j-1})||R_{0i}(\alpha q^{j-k})|\\
                        &\leq |\alpha_i^{-j}P(\alpha q^{-k})P(\alpha q^{-k+1})\cdots P(\alpha q^{-k+j-1})|\, C_{9}^N |a|^{\gamma_1 N^2}|\alpha q^{j-k}|^{N_i}\\
                        &\leq C_{11}^N |a|^{\gamma_1 N^2}|q|^{-k N_i} 
\end{split}
\end{equation}
for all $k=0,1,\ldots$, $i=1,\ldots,m$ and $j=0,\ldots, m$, if $2C_8|\alpha||q|^m< |q|^k$.

Let us define 
\begin{equation}
\gamma := \frac{\log|b|}{\log|a|}.
\end{equation}
Because $|q|= \frac{|a|}{|b|}>1$ and $a,b\in \Z_K$, we have $0\le \gamma <1$.

\begin{lemma}\label{ubounds} Assume that $k$ satisfies the conditions \eqref{k1} and \eqref{k2} and $\gamma$ satisfies the condition
\begin{equation}\label{gamma_cond}
\gamma < \frac{\rho-m\delta}{\rho}.
\end{equation}
Then we have
\begin{equation}
  |p_{jk}|\leq C_{12}^{N}|a|^{\gamma_2 N^2}
\end{equation}
and if $N\ge C_{14}$
\begin{equation} \label{bounds_again}
  |r_{jik}| \leq C_{13}^N |a|^{\gamma_4N^2 - \gamma_5Nn_i},
\end{equation}
where $\gamma_2, \gamma_4$ and $\gamma_5$ are positive constants defined in \eqref{gamma2}, \eqref{gamma4} and
\eqref{gamma5}.
\end{lemma}
\begin{proof}[Proof]
Due to Lemma \ref{last_app} we have  $r_{jik}= p_{jk}f_i(\alpha) - q_{jik}$, where  $r_{jik}= D_k \hat{r}_{jik}$, $p_{jk}= D_k\hat{p}_{jk}$ and 
\[
 |D_k| \le C_ {15}^k |b|^{sk(k+1)/2}|\hat{D}_k|. 
\]
Further by using \eqref{D} and \eqref{D_k} we get
\[
 |D_k| \le C_{16}^{k+N} |a|^{N^2/2 + kN}|b|^{sk(k+1)/2}.
\]
Due to the equation \eqref{hat_p} we have $\hat{p}_{jk} = A_j(\alpha q^{-k})X_k(\alpha, q)$, hence by using \eqref{bound_iterx}, \eqref{bound_A} and \eqref{k1}
we get
\[\begin{split}
  |p_{jk}| &\leq C_{16}^{k+N}|a|^{N^2/2 + kN}|b|^{sk(k+1)/2}C_2^k|q|^{sk(k+1)/2}\max\{1, |\alpha|\}^{C_3 k}C_{10}^N|a|^{\gamma_1 N^2}\\
        &\leq C_{12}^{N}|a|^{(1/2+\rho+\gamma_1+ s\rho^2/2)N^2 }\\
        &=C_{12}^{N}|a|^{\gamma_2 N^2},
\end{split}\]
where 
\begin{equation}\label{gamma2}
\gamma_2 = \gamma_2(\delta, \rho) = \gamma_1(\delta) + \rho + \frac{1}{2}(s\rho^2 + 1).
\end{equation}

We have defined in \eqref{r} that $\hat{r}_{ijk} = (\alpha_i\alpha^s)^{k} q^{uk} R_{ji}(\alpha q^{-k})$. Hence by using \eqref{bound_rem} and \eqref{k1} we get
\[\begin{split}
|r_{jik}| &\leq C_{17}^{k+N}|a|^{N^2/2+kN}|b|^{sk(k+1)/2} C_{11}^N|a|^{\gamma_1 N^2} |q|^{-kN_i} \\
      &\le C_{18}^{k+N} |a|^{(1/2 + \gamma_1)N^2  - k(N_i-N)}|b|^{sk^2/2+ kN_i}\\
      &\leq C_{13}^{N}|a|^{( 1/2 + \gamma_1 + \delta(\rho-m\delta))N^2 -(\rho-m\delta)Nn_i}|b|^{(s\rho/2 + 1-\delta)\rho N^2+\rho Nn_i}\\
     &= C_{13}^{N}|a|^{\gamma_3 N^2-(\rho-m\delta)Nn_i}|b|^{(s\rho/2 +1-\delta)\rho N^2+ \rho N n_i} 
\end{split}\]
if $
N> C_{14}$,
where 
\begin{equation}\label{gamma3}
\gamma_3= \gamma_3(\delta, \rho)= \gamma_1(\delta)+ \frac{1}{2}+\delta(\rho-m\delta).
\end{equation}
Now $|b| = |a|^\gamma$ and we can write the upper bound in the form
\[
|r_{jik}| 
         \le C_{13}^N |a|^{\gamma_4 N^2 - \gamma_5 Nn_i},
\]
where
\begin{align}
\gamma_4 &= \gamma_4(\delta, \rho)= \gamma_3(\delta,\rho) + \gamma\rho(s\rho/2 + 1 - \delta)\label{gamma4},\\
\gamma_5 &= \gamma_5(\delta, \rho)= \rho-m\delta - \gamma \rho. \label{gamma5} 
\end{align}

Constants $\gamma_2$, $\gamma_3$ and $\gamma_4$ are clearly positive, $\gamma_5$ is positive because of \eqref{gamma_cond}.

\end{proof}
In order to get the linear independence measure for numbers $1, f_1(\alpha), \ldots, f_m(\alpha)$, it is essential to have a condition
\begin{equation}\label{exp_cond}
\gamma_5-m\gamma_4 > 0
\end{equation}
The condition \eqref{exp_cond} will be satisfied if we set
\begin{equation}\label{cond-rho}
\rho > \frac{m(\gamma_1+ \frac{1}{2})}{1-m\delta}+m\delta
\end{equation}
and
\begin{equation}\label{gamma_cond2}
\gamma < \frac{\rho - m\delta -m\gamma_3}{\rho+m\rho(\frac{s\rho}{2} +1 - \delta)}.
\end{equation}
Namely, the inequality \eqref{cond-rho} implies that
\[\rho - m\delta > m(\gamma_1 + \frac{1}{2}+ \delta(\rho-m\delta))=m\gamma_3\]
and further
\[\rho - m\delta -m\gamma_3 >0.\]
Hence the condition \eqref{gamma_cond2} seems to be relevantly stated and it implies the  condition \eqref{exp_cond} immediately.


\section{A Linear Independence Measure}
We use the notation
\[L:= l_0 + l_1f_1(\alpha)+ \ldots + l_mf_m(\alpha), \quad (l_0, l_1, \ldots, l_m)\in\Z_K^{m+1}\setminus\{\bar{0}\} \]
for the linear form to be estimated
and denote the linear forms introduced in the previous section in Lemma \ref{last_app} shortly
\[r_{i} = p \,f_i(\alpha) - q_{i}.\]
Now  we get
\begin{equation}\begin{split} \label{estimating}
 p L &= p\, l_0 + l_1 p\, f_1(\alpha) + \ldots + l_m p\, f_m(\alpha) \\
           &= p\,l_0 + l_1 (r_{1}+ q_{1}) + \ldots + l_m (r_{m}+q_{m})\\
           &= G + R,
\end{split}\end{equation}
where
\begin{equation}\label{GG}
G= p\,l_0  + l_1q_{1}  + \ldots + l_m q_{m}
\end{equation}
and
\begin{equation}\label{RR}
R= l_1 r_{1} + \ldots + l_m r_{m}.
\end{equation}

By Lemma \ref{last_app} we know that $G\in \Z_K$. If $G\ne 0$, we get
\[1 \le |G| = |pL - R| \le |p||L| + |R|.\]
Next we shall show that the parameters $j$, $k$ and $\bar{n}=(n_1,\ldots, n_m)$ can be chosen so that $G \ne 0$ and
\begin{equation}\label{puolikas}
|R| < \frac{1}{2}.
\end{equation}
Then we obtain a linear independence measure
\begin{equation}\label{RResult} 
|L| > (2|p|)^{-1}. 
\end{equation}



Suppose that $k$ satisfies the conditions \eqref{k1} and \eqref{k2} and $\gamma_5 - m\gamma_4>0$. Take an arbitrary number $\varepsilon >0$ satisfying
\[\varepsilon < \frac{\gamma_5 - m\gamma_4}{2m},\]
so that we have $\gamma_5 - m(\gamma_4 + 2 \varepsilon) > 0$.
Define
\[S_0 = \max\{\varepsilon^{-1}\log_{|a|}C_{12}, \,\varepsilon^{-1}\log_{|a|}C_{13},\, \varepsilon^{-1}\log_{|a|}2m, C_{14},\, \frac{\gamma_5}{2\varepsilon}, \, \varepsilon^{-1}(m+1)\gamma_5 +1\}.\]
Then according the Lemma \ref{ubounds} and the choice of $S_0$, for every $N \geq S_0$, we have
\begin{align}
|p_{jk}|\le C_{12}^N|a|^{\gamma_2N^2} &\le |a|^{(\varepsilon+\gamma_2) N^2},\label{flim_p} \\
|r_{jik}|\le C_{13}^N|a|^{\gamma_4N^2-\gamma_5Nn_i} &\le |a|^{(\gamma_4 +\varepsilon) N^2 - \gamma_5Nn_i},\label{flim_r} \\
2m & \le |a|^{\varepsilon N},\label{flim_2m}\\
2\varepsilon N &\geq \gamma_5,\label{apulim_gamma5}\\
(m+1)\gamma_5&\leq\varepsilon(N-1).\label{apulim_(m+1)gamma5}
\end{align}

Let us estimate a linear form
\begin{equation}\label{lform}
L = l_0 + l_1f_1(\alpha)+ \ldots + l_mf_m(\alpha)
\end{equation}
with integer coefficients $l_i\in K$ satisfying the condition 
\[H=H_1H_2\cdots H_m > H_0,\] where
$H_i=\max\{1, |l_i|\}$ and
$H_0= |a|^{(S_0+m)^2(\gamma_5 -m(\gamma_4 + 2\varepsilon))}$. 
We obtained earlier that $pL = G+ R$, where $G$ and $R$ are defined in \eqref{GG} and \eqref{RR}. We prove first the condition \eqref{puolikas} for $R$. 

\vspace{1 em}

\begin{lemma}\label{ess}
Positive integers $n_1, \ldots, n_m$ and $N= n_1+ \ldots + n_m$ can be chosen so that
inequalities
\[H_ir_{i} < \frac{1}{2m} \]
hold for every $i=1,\ldots, m$.
\end{lemma}
\begin{proof}[Proof]
By \eqref{flim_r} and \eqref{flim_2m} we know  that for all $N\ge S_0$ we have
\[|r_{i}|\leq |a|^{(\gamma_4 + \varepsilon)N^2 - \gamma_5 n_i N}\]
and $2m \le  |a|^{\varepsilon N}$.
Hence we need to show that the inequalities 
\[|a|^{-(\gamma_5 n_i N - (\gamma_4+ 2\varepsilon)N^2)} < H_i^{-1},\]
or 
\[\gamma_5 n_i N - (\gamma_4 + 2\varepsilon)N^2 > \log_{|a|}H_i,
\]
hold. 
%
%
First we solve the equation
\begin{equation}\label{log-cond}
 \gamma_5 s_i S - (\gamma_4 + 2\varepsilon)S^2 = \log_{|a|}H_i, 
\end{equation}
where $S= \sum_{i=1}^m s_i$.
By summing the equation when $i=1,\ldots,m$, we get
\[(\gamma_5 - m(\gamma_4 + 2\varepsilon))S^2= \log_{|a|}H\]
and 
\[S = \sqrt{\frac{\log_{|a|}H}{\gamma_5 - m(\gamma_4+ 2\varepsilon)}}> S_0 + m\]
(thanks to the definition of $H_0$).
Further from \eqref{log-cond} we get
\[s_i = \frac{\log_{|a|}H_i+ S^2(\gamma_4 + 2\varepsilon)}{S\gamma_5}.\]

Put now
\[n_i = \lfloor s_i \rfloor, \quad i=1,\ldots,m.\]
Since
\[\sum_{i=1}^m s_i = S,\]
we deduce that
\[S_0 < S-m < N=n_1+\ldots + n_m \leq S.\]
In addition 
\[\begin{split}
n_i &> s_i -1 = \frac{\log_{|a|}H_i+ S^2(\gamma_4 + 2\varepsilon)}{S\gamma_5} -1\\
		&\ge \frac{S(\gamma_4+ 2\varepsilon) -\gamma_5 }{\gamma_5}.
		\end{split}\]
By \eqref{apulim_gamma5} and the choice of $N$ we get now		
		\[\begin{split}
   n_i & \ge \frac{\gamma_4}{\gamma_5}N = \frac{\gamma_3 + \gamma\rho(s\rho/2 + 1 - \delta)}{\rho-m\delta-\gamma\rho}N\\
    &\ge \frac{\gamma_3}{\rho-m\delta}N = \frac{ \gamma_1+ \frac{1}{2}+ \delta(\rho-m\delta)}{\rho-m\delta}N> \delta N  
\end{split}\]
as required. Because $N > S-m >S_0$, we get
from \eqref{flim_r} that 
\[|r_{i}|\le |a|^{(\gamma_4 + \varepsilon)N^2 - \gamma_5 n_i N} < |a|^{(\gamma_4 + \varepsilon)S^2 -\gamma_5 n_i (S-m)} \]
and from \eqref{log-cond} that
\[H_i= |a|^{\gamma_5 s_i S - (\gamma_4 + 2\varepsilon)S^2} < |a|^{\gamma_5 (n_i +1)S - (\gamma_4 + 2 \varepsilon)S^2}.\]
Hence by combining these approximations we get
\[\begin{split}H_i|r_{i}| &< |a|^{-\varepsilon S^2 +\gamma_5 S + \gamma_5 m n_i} < |a|^{-\varepsilon S^2 + \gamma_5S(1+m)}\\
                              &= |a|^{S(-\varepsilon S +\gamma_5 (1+m))}.
\end{split}\] 
Now by using \eqref{flim_2m} and \eqref{apulim_(m+1)gamma5} we get the wanted result
\[H_i|r_{i}| <  |a|^{-\varepsilon S} \leq \frac{1}{2m}.\]
%


\end{proof}

By \eqref{RR} and Lemma \ref{ess} we obtain now
\[|R|\leq \sum_{i=1}^m |l_i||r_i|\leq \sum_{i=1}^m H_i|r_i| < \frac{1}{2}.\]
By Lemma \ref{determ_cond} for any given linear form $L$ there exists an index $j\in\{0,1,\ldots, m\}$ such that
\[G_{jk}= l_0p_{jk} + l_1q_{j1k}+ \ldots + l_mq_{jmk}\neq 0.\]
Since by Lemma \ref{last_app} $G= G_{jk}$ is an integer in $K$, we have 
\begin{equation}
|G|\geq 1.
\end{equation}
%
Hence we have the linear independence measure \eqref{RResult} and by \eqref{flim_p} we get
\[|L| > (2|p|)^{-1} \geq \frac{1}{2}|a|^{-(\gamma_2+\varepsilon)N^2}.\]
%
%
Further we get
\[|L| > \frac{1}{2}|a|^{-(\gamma_2 + \varepsilon)S^2} \ge \frac{1}{2}H^{-(\gamma_2+ \epsilon)/(\gamma_5 - m(\gamma_4+ 2\varepsilon))}\]
for all $H> H_0$.

We set now $\delta_0= m\delta$ and $\rho_0 = \rho-m\delta$.
Since $\varepsilon > 0$ is arbitrary, we can state the final result in the following form.
\begin{lause}\label{thm2}
Suppose that none of the functions $f_1(z), \ldots, f_m(z)$ is a polynomial and 
$\alpha_i\neq \alpha_j q^l$ for all $i \neq j$ and $l\in \Z$. 
Let $\alpha$ be a non-zero element of $K$ satisfying $P(\alpha q^{-k})\ne 0$, $k=1,2,\ldots$. Let
\[\gamma_1 = \frac{(2-\delta_0)^2(1-\delta_0)}{2\delta_0}, \quad \gamma_2= \gamma_1 + \rho_0+ \delta_0 + \frac{1}{2}(1+s(\rho_0+ \delta_0)^2),\]
\[\gamma_3=\gamma_1+ \frac{1}{2}+ \frac{\delta_0\rho_0}{m}, \quad \gamma_4 = \gamma_3 + \gamma(\rho_0+ \delta_0)(1-\frac{\delta_0}{m} +s\frac{\rho_0+\delta_0}{2}),\] 
\[\gamma_5= \rho_0-\gamma(\rho_0 +\delta_0) \]
be positive constants, where $0< \delta_0 < 1$ and
\begin{equation}\label{bound_rho}
\rho_0 > \frac{m(\gamma_1 + \frac{1}{2})}{1-\delta_0}>0 .
\end{equation}
Let $\gamma = \frac{\log|b|}{\log|a|}$ satisfy the condition
\begin{equation}\label{bound_gamma}
\gamma < \frac{\rho_0 - m\gamma_3}{(\rho_0+\delta_0)(1+  \frac{sm(\rho_0+\delta_0)}{2} +m - \delta_0)}.
\end{equation}
Then for any $\varepsilon_0 > 0$ there exist a positive constant $C_0= C_0(\varepsilon_0)$ such that for any $(l_0, l_1, \ldots, l_m )\in\Z_K^{m+1}\setminus\{\bar{0}\}$ there holds the inequality
\[|l_0 + l_1f_1(\alpha) + \cdots + l_mf_m(\alpha)|> C_0 (H_1H_2\cdots H_m)^{-\gamma_2/(\gamma_5 -m\gamma_4)-\varepsilon_0},\]
where $H_i =\max\{1, |l_i|\}$ for $i=1,\ldots, m$.
\end{lause}


\section{An Alternative Proof}
In this section we obtain a Baker-type linear independence measure for the numbers $1, f_1(\alpha), \ldots, f_m(\alpha)$, where $\alpha\in K \setminus\{0\}$ and $P(\alpha q^{-k})\ne 0$, by using
Matala-aho's axiomatic Baker-type results \cite{Mat}. First we  denote
\[\bar{n} = (n_1,\ldots,n_m)^T, \quad N= N(\bar{n})= n_1+\ldots+n_m\]
and modify the upper bounds obtained in Lemma
\ref{ubounds} in the form
\begin{align}
|p_{jk}| &\le e^{U(N)},\\
|r_{jik}| &\le e^{-V_i(\bar{n})},
\end{align}  
where \[U(N) = \log|a|(\gamma_2 N^2 + N \log_{|a|}C_{12}):= \gamma_6 N^2 + \gamma_7 N\] and
\[V_i(\bar{n}) = \log|a|(\gamma_5Nn_i - \gamma_4 N^2 - N\log_{|a|}C_{13}):= \gamma_8 Nn_i - \gamma_9 N^2 - \gamma_{10}N.\]
These upper bounds hold for every $N \ge C_{14}$ and if we assume that \eqref{cond-rho} and \eqref{gamma_cond2} hold, we have $\gamma_6, \gamma_7, \gamma_8, \gamma_9, \gamma_{10}\in \R_+$ and $\gamma_8- m\gamma_9 >0$. Furthermore, we suppose that all the assumptions of Theorem \ref{thm2} hold.  Now Theorem and 3.1 and 3.6
in \cite{Mat} imply that for all $(l_0,l_1,\ldots, l_m)\in \Z_K^{m+1}\setminus\{\bar{0}\}$ we have
\begin{equation}
|l_0 + l_1f_1(\alpha) + \ldots + l_m f_m(\alpha)| > F_1\left(\prod_{i=1}^m (2mH_i)\right)^{-\frac{\gamma_6}{\gamma_8-m\gamma_9}-\frac{A}{\sqrt{\log H}}},
\end{equation}
where 
\[H=\prod_{i=1}^m(2mH_i) \ge F_2, \quad H_i =\max\{1, |l_i|\},\]
\[F_2 = e^{C_{14}^2(\gamma_8-m\gamma_9)-C_{14}(m^2\gamma_9+m\gamma_10)-m^2\gamma_{10}},\]
\[
A=\frac{1}{\sqrt{\gamma_8-m\gamma_9}}\left(\frac{\gamma_6(m^2\gamma_9+m\gamma_{10})}{\gamma_8-m\gamma_9} + 2m\gamma_6 + \gamma_7 \right)\]
and
$F_1=\frac{1}{2}e^{-B}$ with
\[\begin{split}
B=& \left(\frac{\gamma_6(m^2\gamma_9 + m\gamma_{10})}{\gamma_8-m\gamma_9} +2m\gamma_6 + \gamma_7\right)\left(\frac{m^2\gamma_9+ m\gamma_{10}}{\gamma_8-m\gamma_9)}+m\sqrt{\frac{\gamma_{10}}{\gamma_8-m\gamma_9}}\right) \\ &+ \frac{m^2\gamma_6\gamma_{10}}{\gamma_8-m\gamma_9} + m^2\gamma_6 + m\gamma_7.
\end{split}\]
Theorem \ref{thm2}.is a corollary of the above result. 


\section{Proof of Theorem 1}
We prove now Theorem \ref{thm1} by using Theorem \ref{thm2}, where we have
a linear independence measure for the numbers $1,f_1(\alpha), \ldots, f_m(\alpha)$. We choose  $Q_i(z)= -P(z) \quad (i=1,\ldots, m)$, which implies that \[f_i(q)= \phi(\alpha_i)\] as we proved  in Section \ref{difference}.

We assume that either  $\deg P(z) < s$ or $\deg P(z)=s$ and $\alpha_i \neq P_sq^n$ ($i=1,\ldots, m$, $n\in\Z_+$), where $P_s$ is the leading coefficient of $P(z)$. These assumptions imply that none of corresponding $f_i(z)$ is not a polynomial in $K[z]$ \,(see details in \cite{AKV}).
Thus 
Theorem \ref{thm2} 
implies the linear independence measure for the numbers $1$, $\phi(\alpha_1), \ldots, \phi(\alpha_m)$ expressed in Theorem \ref{thm1}, where instead of $\mu(m,s)$ is  $\frac{\gamma_2}{\gamma_5-m\gamma_4}$,
 which equals
\[\frac{4-7\delta_0(1-\delta_0) -\delta_0^3 + 2\delta_0\rho_0 +s\delta_0(\rho_0+\delta_0)^2}
 {2\delta_0(1-\delta_0)\rho_0 -m(4-7\delta_0 + 5\delta_0^2-\delta_0^3)-\gamma(\delta_0+\rho_0)\delta_0(2+2m -2\delta_0+sm(\delta_0+\rho_0))}.\]
 We choose now $\delta_0= 1/2$, then 
\[
\frac{\gamma_2}{\gamma_5-m\gamma_4}
       =\frac{4s\rho_0^2 + 4(s+2)\rho_0 + (s+17)}{4\rho_0 -13m - \gamma(4sm\rho_0^2+ 4(2m +ms +1)\rho_0+(4m+ms+2))}
\]
and
\[ \begin{split}
\rho_0(m,s, \gamma)= \frac{13m}{4} + &\frac{\gamma}{2(1-\gamma)} \\
&+\sqrt{\left(\frac{13m}{4} +\frac{\gamma }{2(1-\gamma)}\right)^2 + 
\frac{13m(s+2)+s+17}{4s} + \frac{(s+2)\gamma}{2s(1-\gamma)}}\end{split}\]
admits the minimum value
for the above expression (in the case $\delta_0= 1/2$).
The essential conditions \eqref{bound_rho} and \eqref{bound_gamma} for $\rho_0$ and $\gamma$ become in this case in the form
\begin{equation}\label{fbound_rho}
\rho_0 > \frac{13m}{4}
\end{equation}
and
\begin{equation}\label{fbound_gamma}
\gamma < \frac{4\rho_0-13m}{4ms\rho_0^2 + 4(2m+ms+1)\rho_0+4m+ms+2}.
\end{equation}
The condition \eqref{fbound_rho}  is clearly satisfied. For the latter one we define a function $f:[0,1[\to \R$, 
\[f(\tau) = \frac{4\rho_0(m,s,\tau)-13m}{4ms\rho_0^2(m,s, \tau) + 4(2m+ms+1)\rho_0(m,s, \tau)+4m+ms+2}- \tau.\]
Since $\rho_0(m,s,0) > \frac{13m}{4}$, we see that $f(0)$ is positive. On the other hand \[\lim_{\tau \to 1^-} \rho_0(m,s,\tau) =\infty,\]  hence $\lim_{\tau \to 1^-}f(\tau)= -1$. Because $f$ is continuous it has at least one zero-point in the interval $]0,1[$. We define
\[\Gamma(m,s)= \min\{\tau | f(\tau)=0, \, 0<\tau <1\}.\]
If we choose
\[0 \le \gamma < \Gamma(m,s)\]
the condition \eqref{fbound_gamma} will be satisfied.  This proves the Theorem 1.



\end{document}